\documentclass[11pt, reqno]{amsart}

\author[P.~Leonetti]{Paolo Leonetti}
\address{Universit\`a ``Luigi Bocconi''\\Department of Statistics\\Milan, Italy}
\email{leonetti.paolo@gmail.com}
\urladdr{\url{https://sites.google.com/site/leonettipaolo/}}


\title{A Characterization of Convex Functions}

\usepackage{amsmath}
\usepackage{amssymb}
\usepackage{amsthm}
\usepackage[left=3.5cm, right=3.5cm, paperheight=11.8in]{geometry}
\usepackage{hyperref}
\usepackage{fancyhdr}
\usepackage{enumitem}
\usepackage{graphicx}
\usepackage[utf8]{inputenc}

\newtheorem{thm}{Theorem}

\newtheorem{cor}[thm]{Corollary}

\theoremstyle{definition}

\newtheorem{claim}{\textsc{Claim}}

\hypersetup{
    pdftitle={A characterization of convexity},
    pdfauthor={Paolo Leonetti},
    pdfmenubar=false,
    pdffitwindow=true,
    pdfstartview=FitH,
    colorlinks=true,
    linkcolor=blue,
    citecolor=green,
    urlcolor=cyan
}

\AtBeginDocument{%
   \def\MR#1{}
}
                                 
\begin{document}

\maketitle
\thispagestyle{empty}

\begin{abstract} 
\noindent{} Let $D$ be a convex subset of a real vector space. It is shown that a radially lower semicontinuous function $f: D\to \mathbf{R}\cup \{+\infty\}$ is convex if and only if for all $x,y \in D$ there exists $\alpha=\alpha(x,y) \in (0,1)$ such that $f(\alpha x+(1-\alpha)y) \le \alpha f(x)+(1-\alpha)f(y)$.
\end{abstract}

\section{Introduction.}

Let $X$ be a vector space over the real field and fix a convex set $D \subseteq X$. A function $f: D\to \mathbf{R}$ is convex whenever
\begin{equation}\label{def:convexity}
\textstyle f(\lambda x+(1-\lambda)y) \le \lambda f(x)+(1-\lambda)f(y)
\end{equation}
for all $\lambda \in [0,1]$ and for all $x,y \in D$. In addition, $f$ is said to be
\begin{enumerate}
\item \label{itemalphaconvex} $\alpha$\emph{-convex}, for some fixed $\alpha \in (0,1)$, if
$$
\textstyle f(\alpha x+(1-\alpha)y) \le \alpha f(x)+(1-\alpha)f(y)
$$
for all $x,y \in D$;
\item \label{itemmidconvex} \emph{midconvex} if
$$
\textstyle f\left(\frac{1}{2}(x+y)\right) \le \frac{1}{2}(f(x)+f(y))
$$
for all $x,y \in D$.
\end{enumerate}
It is clear that convex functions are $\alpha$-convex. In addition, it easily follows by the Dar\'{o}czy--P\'{a}les identity, see \cite[Lemma 1]{MR970258}, that $\alpha$-convex functions are midconvex, see also \cite{MR821804} and \cite[Proposition 4]{MR3542948}. Finally, we recall that if $f$ is midconvex then inequality \eqref{def:convexity} holds for all $x,y \in D$ and \emph{rationals} $\lambda \in [0,1]$. 

It is well known that, assuming the axiom of choice, there exist nonconvex functions that are $\alpha$-convex. On the other hand, a \emph{continuous} real function $f$ is convex if and only if it is $\alpha$-convex if and only if it is midconvex. We refer the reader to \cite{MR0442824} and references therein for a textbook exposition on convex functions. 

In this note we provide another characterization of convexity, assuming that the function $f$ is \emph{radially lower semicontinuous}, 
i.e.,
$$
f(x) \le \liminf_{t \,\downarrow\, 0}f(x+t(y-x))
$$
for all $x,y \in D$. 
Our main result is as follows:

\begin{thm}\label{thm:main}
Let $D$ be a convex subset of a real vector space. Then, a radially lower semicontinuous function $f: D\to \mathbf{R}\cup \{+\infty\}$ is convex if and only if, for all $x,y \in D$, there exists $\lambda=\lambda(x,y) \in (0,1)$ that satisfies inequality \eqref{def:convexity}.
\end{thm}

The following corollary, which has been conjectured by Miroslav Pavlovi\'{c} in the case of real-valued continuous functions, is immediate:
\begin{cor}
Let $I\subseteq \mathbf{R}$ be a nonempty interval. Then, a lower semicontinuous function $f: I\to \mathbf{R}\cup \{+\infty\}$ is convex if and only if, for all $x,y \in I$, there exists $\lambda=\lambda(x,y) \in (0,1)$ that satisfies inequality \eqref{def:convexity}.
\end{cor}

It is worth noting that the lower semicontinuous assumption cannot be relaxed too much since there exists a nonconvex function $f: I \to \mathbf{R} \cup \{+\infty\}$ that is lower semicontinous everywhere in $I$ except at countably many points and such that, for all $x,y \in I$, there exists $\lambda=\lambda(x,y) \in (0,1)$ that satisfies inequality \eqref{def:convexity}. Indeed, it is sufficient to choose $I=\mathbf{R}$ and let $f$ be the indicator function of the rationals, i.e., $f(x)=1$ if $x$ is rational and $f(x)=0$ otherwise.

\section{Proof of Theorem \ref{thm:main}.}

We need only to prove the ``if'' part. Moreover, the claim is obvious if $D$ is empty, hence let us suppose hereafter that $D\neq \emptyset$.

Fix $x,y \in D$ and define the set 
$$
S:=\{\lambda \in [0,1]: f(\lambda x+(1-\lambda)y) \le \lambda f(x)+(1-\lambda)f(y)\}.
$$
Notice that $S=[0,1]$ provided that $f(x)=+\infty$ or $f(y)=+\infty$. Hence, let us assume hereafter that $f(x)<+\infty$ and $f(y)<+\infty$.

\begin{claim}\label{claim1}
$S$ is a nonempty compact subset of $[0,1]$.
\end{claim}
\begin{proof}
Note that $\{0,1\}\subseteq S\subseteq [0,1]$, hence $S$ is nonempty and bounded. 

Let $(\lambda_n)$ be a sequence in $S$ such that $\lambda_n \uparrow \lambda> 0$. Define the real sequence $(t_n)$ by $t_n:=1-\frac{\lambda_n}{\lambda}$ for each $n$ and note that $t_n \downarrow 0$. Since $f$ is radially lower semicontinuous, it follows that
\begin{displaymath}
\begin{split}
f(\lambda x+(1-\lambda)y) &\le \liminf_{n\to \infty}f(\lambda x+(1-\lambda)y)+t_n(y-\lambda x-(1-\lambda)y)))\\
&= \liminf_{n\to \infty}f(\lambda_n x +(1-\lambda_n)y)  \\
&\le \liminf_{n\to \infty}\lambda_n f(x)+(1-\lambda_n)f(y)=\lambda f(x)+(1-\lambda)f(y).
\end{split}
\end{displaymath}

Similarly, given a sequence $(\lambda_n)$ in $S$ such that $\lambda_n \downarrow \lambda< 1$, let $(s_n)$ be the sequence defined by $s_n:=1-\frac{1-\lambda_n}{1-\lambda}$ for each $n$ and note that $s_n \downarrow 0$. Then
\begin{displaymath}
\begin{split}
f(\lambda x+(1-\lambda)y) &\le \liminf_{n\to \infty}f(\lambda x+(1-\lambda)y)+s_n(x-\lambda x-(1-\lambda)y)))\\
&= \liminf_{n\to \infty}f(\lambda_n x+(1-\lambda_n)y) \le \lambda f(x)+(1-\lambda)f(y).
\end{split}
\end{displaymath}
Therefore $S$ is also closed, proving the claim.
\end{proof}

\begin{claim}\label{claim2}
$(\lambda_1,\lambda_2) \cap S \neq \emptyset$ for all $\lambda_1,\lambda_2 \in S$ with $\lambda_1<\lambda_2$.
\end{claim}
\begin{proof}
Fix $\lambda_1, \lambda_2 \in S$ with $\lambda_1<\lambda_2$ (note that this is possible since $\{0,1\}\subseteq S$) and define
$$
a = \lambda_1 x + (1-\lambda_1)y\,\,\,\text{ and }\,\,\,b = \lambda_2 x + (1-\lambda_2)y.
$$
Hence $a,b \in D$ and, by hypothesis, there exists $\lambda=\lambda(a,b)\in (0,1)$ such that
$$
f(\lambda a+(1-\lambda)b) \le \lambda f(a)+(1-\lambda)f(b).
$$
At this point, define $\lambda^\prime := \lambda \lambda_1 + (1-\lambda) \lambda_2$ and observe that
\begin{displaymath}
\begin{split}
f(\lambda^\prime x+(1-\lambda^\prime)y)&=f(\lambda a+(1-\lambda)b) \le \lambda f(a)+(1-\lambda)f(b)\\
&\le \lambda (\lambda_1 f(x)+(1-\lambda_1)f(y))+(1-\lambda)(\lambda_2 f(x) +(1-\lambda_2)f(y)) \\
&=\lambda^\prime f(x)+(1-\lambda^\prime) f(y).
\end{split}
\end{displaymath}
Therefore $\lambda^\prime \in S$. The claim follows by the fact that $\lambda_1<\lambda^\prime<\lambda_2$.
\end{proof}

To complete the proof of Theorem \ref{thm:main}, let us suppose for the sake of contradiction that 
$$
\Lambda:=[0,1] \setminus S\neq \emptyset.
$$
Note that $\Lambda$ is open since it can be written as $(0,1) \cap S^c$ and $S^c$ is open by Claim \ref{claim1}. Fix $\lambda \in \Lambda$. Hence, there is a maximal open interval $(\lambda_1,\lambda_2)$ containing $\lambda$ and contained in $\Lambda$. 
Then $\lambda_1, \lambda_2 \in S$ while $(\lambda_1,\lambda_2) \cap S = \emptyset$, contradicting Claim \ref{claim2}. 

This shows that $f$ is necessarily convex. 

\subsection*{Acknowledgments.}
The author is grateful to the editor and two anonymous referees for their remarks that allowed a substantial improvement of the presentation.

\bibliographystyle{amsplain}


\end{document}